\makeatletter
\IfFileExists{manuscript/math_macros.tex}{%
  \def\input@path{{manuscript/}}%
  \def\RootBuild{1}%
}{}
\makeatother
\RequirePackage{fix-cm}
\documentclass[twoside]{article}

\usepackage[preprint]{paperstyle}
\ifdefined\fancyfootinit
  \fancyfootinit{\renewcommand{\footruleskip}{0pt}}
\fi
\usepackage{array}
\usepackage{longtable}
\usepackage[T1]{fontenc}
\usepackage{iftex}
\ifPDFTeX
  \usepackage[utf8]{inputenc}
\fi
\usepackage{microtype}
\usepackage[numbers,sort&compress]{natbib}
\usepackage{amsmath}
\usepackage{amssymb}
\usepackage{bbm}
\usepackage{amsthm}
\usepackage{mathtools}
\usepackage{bm}
\usepackage[ruled,vlined]{algorithm2e}

\SetAlgoNlRelativeSize{0}
\SetNlSty{AlgorithmLineNumberStyle}{}{}
\usepackage{graphicx}
\usepackage{capt-of}
\usepackage{stfloats}
\usepackage{booktabs}
\usepackage{xcolor}
\definecolor{darkforest}{rgb}{0.13,0.55,0.13}

\usepackage{wrapfig}
\usepackage{threeparttable}

\usepackage[
    colorlinks=true,
    linkcolor=blue,
    citecolor=blue,
    urlcolor=blue
]{hyperref}
\hypersetup{
    pdftitle={First-Order Primal-Dual Repairs for Scalable Branch and Bound},
    pdfauthor={Francesco Pivi, Jiachang Liu, Andrea Lodi}
}
\usepackage[capitalise,noabbrev]{cleveref}
\usepackage{enumitem}

\newcommand{\R}{\mathbb{R}}

\newcommand{\Z}{\mathbb{Z}}

\newcommand{\bbR}{\mathbb{R}}
\newcommand{\bbN}{\mathbb{N}}

\newcommand{\ba}{\bm{a}}
\newcommand{\bb}{\bm{b}}

\newcommand{\bI}{\bm{I}}

\newcommand{\br}{\bm{r}}

\newcommand{\bu}{\bm{u}}

\newcommand{\bw}{\bm{w}}
\newcommand{\bx}{\bm{x}}
\newcommand{\by}{\bm{y}}
\newcommand{\bz}{\bm{z}}

\newcommand{\bbeta}{\bm{\beta}}

\newcommand{\bgamma}{\bm{\gamma}}

\newcommand{\bnu}{\bm{\nu}}

\newcommand{\bzeta}{\bm{\zeta}}

\newcommand{\bA}{\bm{A}}

\newcommand{\bK}{\bm{K}}

\newcommand{\bW}{\bm{W}}

\newcommand{\calB}{\mathcal{B}}
\newcommand{\calC}{\mathcal{C}}
\newcommand{\calD}{\mathcal{D}}

\newcommand{\calJ}{\mathcal{J}}

\newcommand{\calL}{\mathcal{L}}

\newcommand{\calN}{\mathcal{N}}

\DeclareMathOperator*{\argmin}{arg\,min}

\DeclareMathOperator{\dom}{dom}

\DeclareMathOperator{\prox}{prox}

\newcommand{\ip}[2]{\left\langle #1,#2\right\rangle}

\newcommand{\norm}[1]{\left\lVert #1\right\rVert}

\newcommand{\vzero}{\bm{0}}

  {\par\medskip\noindent\textbf{#1.}\itshape}{\par\medskip}

\newtheorem{theorem}{Theorem}

\newtheorem{lemma}{Lemma}

\theoremstyle{remark}
\newtheorem{remark}{Remark}

\begin{document}
\onecolumn
\raggedbottom
\suppressfloats[t]

\runningtitle{First-Order Primal-Dual Repairs for Scalable Branch and Bound}
\runningauthor{Pivi, Liu, Lodi}

\noindent\begin{minipage}{\textwidth}

\papertitle{First-Order Primal-Dual Repairs for Scalable Branch and Bound}

\newcommand*\samethanks[1][\value{footnote}]{\footnotemark[#1]}

\paperauthor{ Francesco Pivi$^{1,3*}$ \hspace{3em} Jiachang Liu$^{2,4,5,6*}$ \hspace{3em} Andrea Lodi$^{2,3}$ }
\par\vspace{10pt}

{\small
\paperaddress{
    $^{1}$University of Bologna, Italy\\[4pt]
    $^{2}$School of Operations Research and Information Engineering, Cornell University, Ithaca, USA\\[4pt]
    $^{3}$Jacobs Technion--Cornell Institute, Cornell Tech and Technion--IIT, New York, USA\\[4pt]
    $^{4}$Center for Data Science for Enterprise and Society, Cornell University, Ithaca, USA\\[4pt]
    $^{5}$Department of Computer Science and Engineering, Pennsylvania State University, USA\\[7pt]
    $^{6}$Institute for Computational and Data Sciences, Penn State University, USA
}}
\end{minipage}\par\vspace{14pt}

\begingroup
\renewcommand{\thefootnote}{\fnsymbol{footnote}}
\footnotetext[1]{These authors contributed equally.}
\endgroup

\setlength{\parskip}{7pt}
\begin{abstract}

Mixed-integer programs have found applications in many real-world problems but remain difficult to solve at scale.
To improve scalability, there has been increasing interest in incorporating first-order methods as subroutines for solving node relaxations, with success primarily on problems involving specific types of constraints.
However, when more general constraints are considered, first-order methods cannot be directly applied because, at finite iterations, they typically produce infeasible primal and dual solutions, preventing their direct integration into branch-and-bound.
To overcome this limitation, we introduce repair procedures that extract valid primal and dual bounds from infeasible first-order iterates.
These bounds enable node pruning, incumbent updates, and early stopping at any iteration.
Empirically, our approach achieves one to two orders of magnitude speedups when solving real-world instances (with another order of magnitude speedup when running on GPUs) and certifies optimality on instances that are difficult for commercial solvers.

\end{abstract}

\section{Introduction}
\label{sec:introduction}

Many real-world applications require modeling with discrete variables and domain-specific constraints.
In medicine, operations research, and finance, such problems include creating risk scoring systems under interpretability and clinical constraints~\citep{liu2025scalable,liu2022fasterrisk,ustun2019learning}, selecting facility locations under capacity, cost, and routing constraints~\citep{fischetti2016benders,prodhon2014survey}, and allocating capital subject to risk, liquidity, and diversification constraints~\citep{bonami2009exact,kolm201460}.
When life, time, and money are at great risk, we want to solve these problems to high quality or even optimality.

\begin{figure}[t]
\centering
\includegraphics[width=0.68\linewidth]{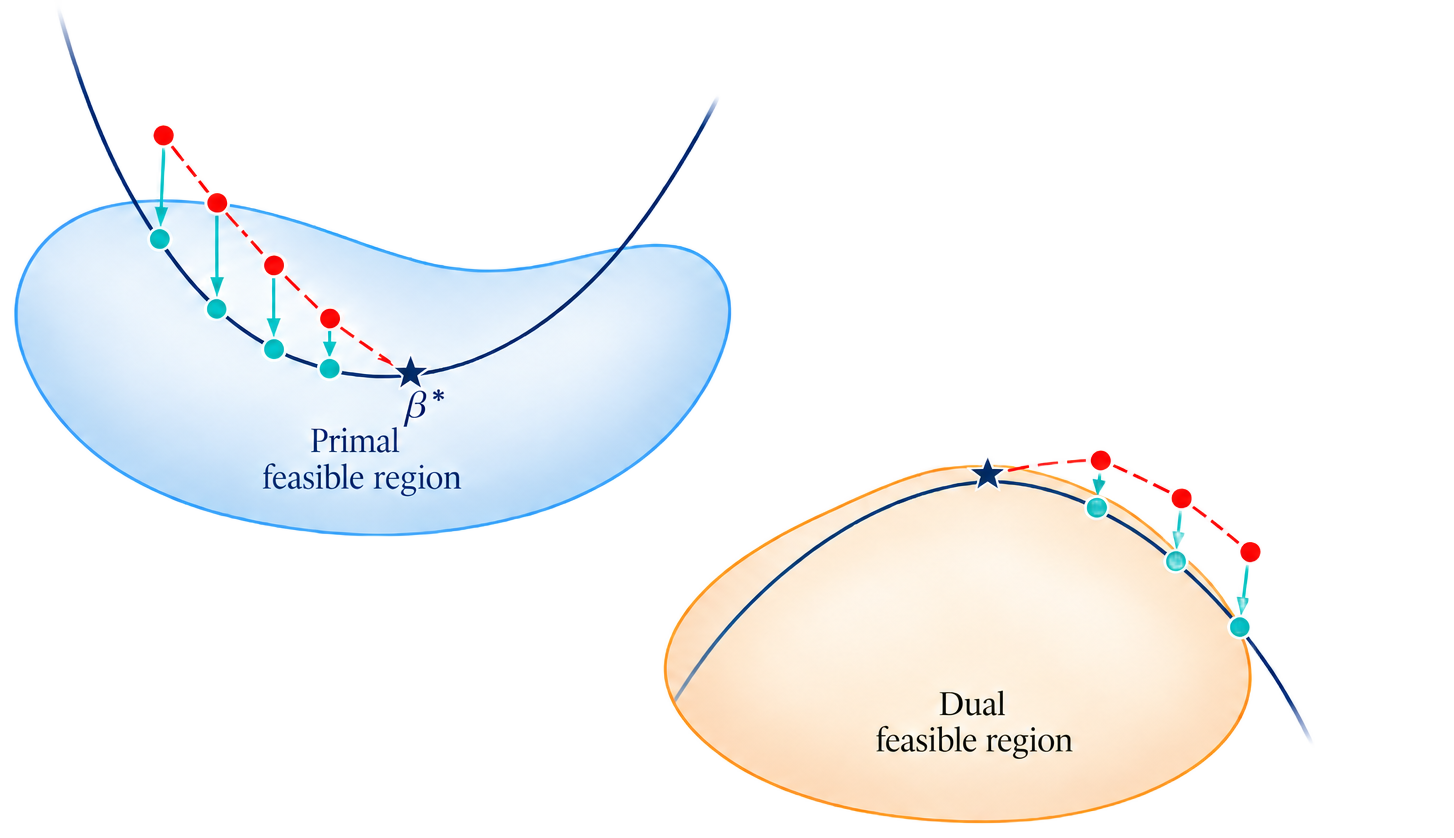}
\caption{Primal and dual-variable repair idea.}
\label{fig:graphical-abstract}
\end{figure}

All of the above problems can be formulated as mixed-integer nonlinear programs (MINLPs).
In this work, we aim to solve the following MINLP:
\begin{equation}
\begin{aligned}
\mathrm{OPT} :=
\inf_{\bbeta \in \R^{p}} \quad & f(\bA \bbeta) + g(\bbeta)\\
\text{s.t.}\quad
& h_l(\bbeta) \leq 0, && l \in [m], \\
& \beta_j \in \Z, && j \in \calJ,
\end{aligned}
\label{eq:minlp}
\end{equation}
where $\bA \in \bbR^{n
\times p}$ is the feature matrix,
$\bbeta$ is a vector containing both continuous and discrete variables (with
$\calJ$ indexing the discrete ones), $[m] := \{1,\ldots,m\}$, and $f$, $g$, and $h_l$'s are the loss function, regularization function, and constraint function, respectively.
All functions are assumed to be proper, closed, and convex.
We further assume that $f$ is $L_f$-smooth, and that the required proximal, projection, and Fenchel-conjugate/support-function evaluations are efficient.
In this work, we only focus on inequality constraints.

Although all functions in Problem~\eqref{eq:minlp} are convex, the integrality constraint $\beta_j\in\mathbb Z$, $j\in\calJ$ makes the problem nonconvex and contain many local minima.
Certifying global optimality therefore requires exploring the discrete structure via a branch-and-bound (BnB) framework.
However, existing BnB methods struggle to scale because they cannot handle~\emph{a broad class of constraints} and~\emph{a large number of variables} simultaneously:
commercial solvers based on outer approximation can handle many types of constraints but stall when the number of variables grows,
whereas recent customized solvers based on first-order methods (FOMs) can scale gracefully with the variable dimension but are limited to only specific types of constraints, such as cardinality and uniform box constraints.

Our goal is to develop a generic FOM-based BnB solver capable of handling a broad class of constraints.
Although it is not new for FOMs to solve constrained convex problems, when we use them to solve node relaxations in the context of BnB, we face two key challenges:
(1) FOMs may only yield infeasible primal and dual iterates, which prevent us from pruning nodes and updating the incumbent (best feasible solution).
(2) FOMs can take too many iterations to converge to high-precision solutions, making the use of FOMs computationally inefficient.

To overcome these challenges, we propose two fixes.
For the dual iterates, we construct safe lower bounds by either repairing the dual equality violation or taking account of this violation explicitly into the dual objective function.
For the primal iterates, we restore feasibility by pulling the primal iterates toward interior points.
Figure~\ref{fig:graphical-abstract} illustrates these primal and dual repairs.
With the help of these techniques, we can prune nodes, update the incumbent, and do early stopping of FOMs, at any iteration.
To the best of our knowledge, this is the first work to enable FOMs to be used for general MINLPs.

To validate the effectiveness of our method, we evaluate our FOM-based BnB solver on both synthetic and real-world datasets, including tasks of sparse regression and classification, facility allocation, and portfolio optimization.
Compared with state-of-the-art commercial solvers, we achieve up to two orders of magnitude speedups and also close the optimality gap on difficult datasets, which is out of reach for existing methods.
Lastly, our algorithm can be easily accelerated on GPUs, bringing another order of magnitude speedup on large-scale instances.

In summary, our contributions are
\begin{enumerate}
    \item We propose methods to extract valid primal and dual bounds from infeasible primal and dual iterates, enabling FOMs to be used in BnB for general MINLP problems.
    \item For our valid dual bounds, we show that they converge at the same rate as that of the residual term of the underlying FOM.
    Our theoretical analysis applies to a broad class of FOMs.
    \item Across both synthetic and real-world benchmarks, our FOM-based BnB framework is up to two orders of magnitude faster than commercial solvers, with an additional order-of-magnitude speedup when running on GPUs.
\end{enumerate}

\section{Related Works}
\label{sec:related-works}

\paragraph{Lower Bound for Pruning.}
In order to prune nodes in BnB, we need to compute the lower bound at each node.
To do this, commercial solvers rely on outer approximations, which linearize nonlinear objectives and constraints with cutting planes \citep{duran1986outer,bonami2008algorithmic,coey2020outer}.
However, this iterative process of adding cuts and solving linear programs can become costly at scale.
To make the computation more efficient, FOMs have been employed to solve node relaxations and
derive dual bounds \citep{hazimeh2022sparse,liu2026gpu,liu2025scalable, naik2017embedded,stellato2018embedded}.
However, existing work can only handle
specific types of constraints.
Our work extends this capability by allowing FOMs to handle a broad class of constraints.

\paragraph{Primal and Dual Feasibility.}
A key step in deriving valid primal and dual bounds is to first obtain feasible primal and dual solutions.
When the entire problem can be written in an unconstrained composite form, $f+g$, past works~\citep{rockafellar1970convex,hazimeh2022sparse,liu2026gpu,fercoq2015mind,ndiaye2017gap} use the proximal operator on $g$ and the gradient of $f$ to obtain feasible primal and dual iterates, respectively.
However, when multiple constraints are present, FOMs~\citep{boyd2011distributed,chambolle2011first,ryuyin2022largescale} typically produce infeasible primal and dual iterates.
To overcome this challenge, we introduce repair methods that allow us to obtain valid primal and dual bounds.
Moreover, unlike the recent radial methods~\citep{renegar2016efficient,grimmer2018radial}, which require changing the underlying algorithm, our method is a plug-and-play step that leaves the backbone algorithm unchanged.

\paragraph{Convergence Rates.}
Recent work has studied the convergence rate of valid lower bounds in the setting of unconstrained composite optimization and proximal gradient methods.
When additional constraints are present, primal-dual splitting methods are more appropriate; however prior works have focused on the convergence rate of the primal-dual gap and fixed-point residual of the Lagrangian formulation~\cite{chambolle2011first,davis2015convergence,dunner2016primal}.
This leaves open the question of how quickly the valid lower bound converges.
In this work, we show this quantity converge at least at the same rate as the residual term.

\paragraph{GPU-accelerated Optimization.}
There is a growing interest in designing GPU-friendly optimization methods and software, by leveraging matrix-vector products, elementwise computations, and proximal updates, while avoiding costly linear-system solves and matrix factorizations \citep{applegate2021practical,lu2025cupdlp,lin2025pdcs,liu2026gpu}.
Among primal-dual splitting methods, we choose the Condat--V\~u algorithm as our base building block because it avoids solving linear systems.
Moreover, when computing the dual bounds of quadratic loss function, we also avoid performing the costly matrix factorization step by exploiting the Fenchel-Young inequality.

\section{Preliminaries}
\label{sec:problem-preliminaries}

\subsection{Branch-and-bound and Node Relaxation}
\label{subsec:convex-relaxation}

We solve Problem~\eqref{eq:minlp} to optimality via a BnB framework, which recursively partitions the feasible region, organizes the resulting subproblems as nodes of a search tree, computes a lower bound for each node, and prunes nodes whose lower bounds exceed the objective value of the incumbent (best feasible solution found so far).
Once all nodes have been processed, the incumbent becomes the global optimal solution.

More specifically, each partitioned region is defined by box constraints on the discrete variables
\begin{equation}
\calB_{\calN}
:=
\left\{
\bbeta\in\R^{p}:
\beta_j \in [l_j,u_j],
\, j\in\calJ
\right\},
\label{eq:node-domain}
\end{equation}
where $\calN$ is the node index, and $l_j\in\R\cup\{-\infty\}$ and $u_j\in\R\cup\{+\infty\}$ denote the lower and upper bounds on $\beta_j$.
For notational simplicity, we omit the node index $\calN$ on $l_j$ and $u_j$ wherever it is clear from context.
From each parent node, we produce two child nodes by tightening one of the intervals in the box constraints.
Specifically, the intervals for the two child nodes are $[c_j,u_j]$ and $[l_j, c_j]$ with $l_j < c_j < u_j$.

The effectiveness of the BnB method is heavily based on the ability of \emph{pruning} nodes in the tree, without resorting to full enumeration. A node -- and the entire subtree rooted on it -- can be safely pruned if it is shown not to contain any improving solution with respect to the incumbent.
For that purpose, one needs to compute at each node a lower bound, which is done by dropping the integrality requirements.
The resulting node relaxation can be written as
\begin{equation}
\begin{aligned}
p_{\calN}^\star \; := \;
\inf_{\bbeta \in \bbR^p} \quad & f(\bA \bbeta) + g(\bbeta)\\
\text{s.t.} \quad & \bbeta \in \calB_{\calN}, \\
& h_l(\bbeta) \leq 0, \; l \in [m]. \\
\end{aligned}
\label{eq:node-relaxation}
\end{equation}

\subsection{First-order Primal-dual Splitting Method}

To solve the node relaxation, we first rewrite Problem~\eqref{eq:node-relaxation} in an equivalent form.
Let $\calC_0:=\calB_{\calN}$ and $ \calC_l:=\left\{\bbeta\in\R^p:h_l(\bbeta)\leq 0\right\}$ for $l\in[m]$, and $\delta_{\calC}$ be the indicator function of a set $\calC$,~\textit{i.e.}, $ \delta_{\calC}(\bgamma) = 0 $ if $\bgamma \in \calC$ and $\delta_{\calC}(\bgamma)=+\infty$ otherwise.

Then, Problem~\eqref{eq:node-relaxation} is equivalent to
\begin{align}
p_{\calN}^\star = \inf_{\bbeta} p(\bbeta)
\label{problem:relaxation_in_composite_form}
\end{align}
where
\begin{align}
p(\bbeta) &:= f(\bA\bbeta)+g(\bbeta)+h(\bK\bbeta),
\label{eq:relaxation_in_primal_composite_form_objective_function} \\
\bK &:=
\begin{bmatrix}\bI_p\\\vdots\\\bI_p\end{bmatrix}
\in\R^{(m+1)p\times p},\\
h(\bgamma_0,\ldots,\bgamma_m) &:=\sum_{l=0}^m\delta_{\calC_l}(\bgamma_l).
\end{align}

A FOM can solve such generalized composite forms by tackling its saddle-point formulation
\begin{align}
\label{problem:relaxation_in_saddle_point_formulation}
p_{\calN}^\star = \inf_{\bbeta} \; \sup_{\bnu}\calL(\bbeta,\bnu),
\end{align}
where
\begin{align}
\calL(\bbeta,\bnu) :=f(\bA\bbeta)+g(\bbeta)+\ip{\bnu}{\bK\bbeta}-h^*(\bnu),
\end{align}
and $h^*(\bnu) := \sup_{\bgamma} \bnu^\top \bgamma - h(\bgamma)$ is known as the Fenchel conjugate of $h$.

Under the assumptions that the proximal operators of both $g$ and $h$ are easy to evaluate, we can solve Problem~\eqref{problem:relaxation_in_saddle_point_formulation} using the primal-dual splitting method.
One such example is the Condat-V\~u algorithm, which performs the following primal and dual updates at each iteration, respectively:
\begin{align}
(\bbeta^+,\bnu^+) = T\left(\bbeta,\bnu\right),
\label{eq:condat-vu-algorithm}
\end{align}
where
\begin{align*}
\bbeta^+
&=\prox_{\sigma g}\!\Bigl(
\bbeta-\sigma\bigl(\bA^\top\nabla f(\bA\bbeta) + \bK^\top\bnu\bigr)\Bigr),
\\
\bnu^+
&=\prox_{\tau h^*}\!\Bigl(
\bnu+\tau\bK(2\bbeta^+ -\bbeta)\Bigr),
\end{align*}
$\sigma>0$ and $\tau>0$ are the primal and dual step sizes, respectively, and $\prox$ is the proximal operator defined as $\prox_{\sigma g}(\bgamma) := \argmin_{\bbeta} \frac{1}{2} \lVert \bgamma - \bbeta \rVert_2^2 + \sigma g(\bbeta)$.

The fixed-point iteration in Equation~\eqref{eq:condat-vu-algorithm} converges to an optimal solution in the $M$-induced normed space, where
\begin{align}
M :=
\begin{bmatrix}
\sigma^{-1}\bI_p & -\bK^\top\\
-\bK & \tau^{-1}\bI_{(m+1)p}
\end{bmatrix},
\end{align}
$\ip{\bw}{\bgamma}_M:=\bw^\top M \bgamma$, and $\norm{\bw}_M:=\sqrt{\ip{\bw}{\bw}_M}$.

For a primer on the primal-dual splitting method and the Condat-V\~u algorithm, please see Appendix~\ref{sec:base-algorithm}.

\subsection{Safe Lower Bounds in Branch and Bound}

Approximate optimality with strict tolerances can be slow to reach, and valid lower bounds on  $p_{\calN}^\star$ might be not guaranteed. This is because their intermediate iterates may violate primal or dual feasibility. This is problematic for three reasons.

First, FOMs can produce an infeasible primal iterate $\bbeta^t$, which we cannot plug into $p(\bbeta^t)$ to obtain a finite value.
Second, even when FOMs produce a feasible primal iterate $\bbeta^t$, its objective value satisfies $p(\bbeta^t)\geq p_{\calN}^\star$ and therefore does not generally provide a valid lower bound.
Third, by weak duality, any feasible dual point provides a valid lower bound, even if it is not optimal.
However, the dual points constructed from intermediate iterates in FOMs may violate the dual constraints.
Their objective values are therefore not guaranteed to lie below $p_{\calN}^\star$ and cannot be used directly as lower bounds for node pruning.

To prune nodes at each iteration $t$, we need a lower bound $d^t$ that is \emph{safe}, i.e., $d^t\leq  p_{\calN}^\star$ and \emph{fast} to compute.
Computing such safe lower bounds in presence of general convex and structured constraints is one of the key motivations in this paper.
 Likewise, we also would like to obtain~\emph{safe upper bounds} $p^t$ ($p_{\calN}^\star \leq p^t $), even when FOMs produce infeasible $\bbeta^t$ so that whenever $p^t \leq U^{\text{incum}}$, where $U^{\text{incum}}$ is the objective value of the incumbent, we can do early stopping for FOMs to save computations.

\section{Methodology}
\label{sec:methodology}

\subsection{Constrained Dual Composite Form}

To derive safe lower bounds, we first rewrite Problem~\eqref{problem:relaxation_in_composite_form} in its dual composite form by introducing two additional dual variables $\bzeta$ and $\bu$, namely
\begin{align}
d^\star
:=
\left\{
\begin{aligned}
\sup_{\bzeta,\bu,\bnu}\quad & d(\bzeta,\bu,\bnu) \\
\mathrm{s.t.}\quad          & V(\bzeta,\bu,\bnu)=\vzero
\end{aligned}
\right\},
\label{problem:dual_formulation}
\end{align}
where
\begin{align}
d(\bzeta,\bu,\bnu) :&= -f^*(\bzeta) -g^*(\bu) -h^*(\bnu), \\
h^*(\bnu_0, \dots, \bnu_m) :&= \sum_{l=0}^m \sigma_{\calC_l}(\bnu_l), \\
V(\bzeta, \bu, \bnu) :&= \bA^\top \bzeta +\bu +\bK^\top\bnu,
\end{align}
and $\sigma_{\calC}(\bu) := \sup_{\bx\in\calC}\ip{\bu}{\bx}$ is known as the support function and is equal to the Fenchel conjugate of the indicator function,~\textit{i.e.}, $\delta_{\calC}^*(\bu) = \sigma_{\calC}(\bu)$.

By the weak duality, every feasible tuple $(\bzeta, \bu, \bnu)$ yields a safe lower bound through $d(\bzeta,\bu,\bnu)\leq d^\star\leq p_{\calN}^\star$.
However, it is difficult for FOMs to produce a sequence of dual iterates that satisfy the linear equality $V(\bzeta, \bu, \bnu) = 0$~\emph{exactly}.
If we evaluate $d(\bzeta, \bu, \bnu)$ at an infeasible dual iterate, this does not produce a safe lower bound.
We next consider how to compute safe lower bounds under this limitation of FOMs.

\subsection{Safe Lower Bound Computation}

\begin{algorithm}[tbp]
\caption{Compute Safe Lower bound}
\label{alg:compute_safe_lower_bound}
\LinesNumbered
\SetCommentSty{AlgorithmCommentStyle}
\DontPrintSemicolon
\KwIn{Primal and dual iterates $(\bbeta^t, \bnu^t)$, (optional) convex and compact set $\calD$ such that $\dom p \subseteq \calD$. }
\KwOut{Safe lower bound $LB^t$ for Problem~\eqref{eq:node-relaxation}.}
\tcp{Perform one-step Condat-V\~u update}
$\bbeta^+, \bnu^+ \gets T (\bbeta^t, \bnu^t)$ \label{alg_line:condat_vu_update}\;
$\bzeta^+ \gets \nabla f\left(\bA \bbeta^t \right)$ \label{alg_line:dual_variable_zeta_update}\;
$\bu^+ \gets \frac{1}{\sigma}\left(\bbeta^t-\bbeta^+\right) - \bA^\top\bzeta^+-\bK^\top\bnu^t$ \label{alg_line:dual_variable_u_update}\;
\If{$\calD$ is not provided}{
    \tcp{Method 1: repair dual variable}
    $\tilde{\bu} \gets - \bA^\top \bzeta^+ - \bK^\top \bnu^+$ \label{alg_line:repair_dual_variable_u}\;
    $LB^t \gets d(\bzeta^+, \tilde{\bu}, \bnu^+)$\label{alg_line:compute_safe_lower_bound_repair}\;
}
\Else{
    \tcp{{Method 2: penalize dual constraint violation; $\sigma_{\calD} (\cdot)$ is the support function over $\calD$}}
    $LB^t \gets d(\bzeta^+, \bu^+, \bnu^+) - \sigma_{\calD} \left(-V\left(\bzeta^+, \bu^+, \bnu^+ \right) \right)$ \label{alg_line:compute_safe_lower_bound_penalize}\;
}
\KwRet{$LB^t$}\;

\end{algorithm}

We propose two alternative ways to calculate safe lower bounds, which are shown in Algorithm~\ref{alg:compute_safe_lower_bound}.
At the start (Lines~\ref{alg_line:condat_vu_update}-~\ref{alg_line:dual_variable_u_update}), Algorithm~\ref{alg:compute_safe_lower_bound} produces an auxiliary dual tuple $\left(\bzeta^+, \bu^+, \bnu^+ \right)$.
By construction, we can show that each individual dual variable is feasible for its dual function,~\textit{i.e.}, $\bzeta^+ \in \dom f^*$, $\bu^+ \in \dom g^*$, and $\bnu^+ \in \dom h^*$.
However, this construction does not guarantee $V(\bzeta^+, \bu^+, \bnu^+) = 0$, so we cannot use $d(\bzeta^+, \bu^+, \bnu^+)$ to compute a safe lower bound.
We fix this issue by either repairing the dual variable $\bu$ (Lines~\ref{alg_line:repair_dual_variable_u}--\ref{alg_line:compute_safe_lower_bound_repair}) or adding a penalty term based on the constraint violation (Line~\ref{alg_line:compute_safe_lower_bound_penalize}). The repaired bound is finite only if $g^*(\tilde{\bu})<+\infty$; otherwise it is $-\infty$.

The safe lower bounds in Algorithm~\ref{alg:compute_safe_lower_bound} actually enjoy several nice properties.
In the following theorem, we show that
1) these are valid lower bounds,~\textit{i.e.}, $\text{LB}_{\text{repair}}^t \leq d^\star$ and $\text{LB}_{\text{violation}}^t \leq d^\star$;
2) if $(\bbeta^t, \bnu^t)$ converge to an optimal solution, then the safe lower bounds converge to $d^\star$ as $t \rightarrow +\infty$;
3) their convergence rates are dominated by the residual error $\lVert (\bbeta^{t+1}, \bnu^{t+1}) - (\bbeta^t, \bnu^t)\rVert_M$.

\begin{theorem}
\label{theorem:convergence_rate}
    Let $\{\bz^t\}_{t \in \bbN}$, where $\bz^t = \left(\bbeta^t, \bnu^t \right)$, be a sequence of iterates produced by a primal-dual splitting method in solving Problem~\eqref{problem:relaxation_in_saddle_point_formulation}.
    Let $e^t := \norm{\br^t}_M := \lVert (\bI - T) \left( \bz^t \right) \rVert_M$ be the residual error at each iteration.
    Let $s:=\sigma^{-1}-\tau\norm{\bK}_2^2>0$, and choose $\sigma,\tau>0$ so that $s\geq L_f \|\bA\|_2^2/2$.
    Suppose there exists some $R_{\bz} > 0$ and some optimal solution $\bz^\star = \left( \bbeta^\star, \bnu^\star \right)$ such that $\lVert \bz^t - \bz^* \rVert_M \leq R_{\bz}$ for $t \in \bbN$.

    (Assumption 1) If there exists a $L_{g^*} > 0$ such that $g^*$ is $L_{g^*}$-Lipschitz, then we have
    \begin{equation*}
\begin{aligned}
0&\leq\Delta_{\text{repair}}\left(\bz^t\right) :=d^\star-\text{LB}_{\text{repair}}^t\\
&\leq\left(R_{\bz}
+\frac{L_{g^*}+\lVert\bbeta^\star\rVert_2}{\sqrt{\sigma}}\right)e^t
+\frac{L_f\lVert\bA\rVert_2^2}{2s}\left(e^t\right)^2.
\end{aligned}
\end{equation*}

    (Assumption 2) If there exists a nonempty compact and convex set $\calD$ such that $\dom p \subseteq \calD$ and let $R_{\bbeta}:= \sup_{\bbeta \in \calD} \lVert \bbeta - \bbeta^\star \rVert_2$, then we have
    \begin{equation*}
\begin{aligned}
0&\leq\Delta_{\text{violation}}\left(\bz^t\right) :=d^\star-\text{LB}_{\text{violation}}
^t\\
&\leq\left(R_{\bz}+\frac{R_{\bbeta}}{\sqrt{\sigma}}\right)e^t
+\frac{L_f\lVert\bA\rVert_2^2}{2s}\left(e^t\right)^2.
\end{aligned}
\end{equation*}

\end{theorem}

\begin{remark}
Algorithm~\ref{alg:compute_safe_lower_bound} and Theorem~\ref{theorem:convergence_rate} are applicable for a broad class of primal-dual splitting methods.
For example, the primal and dual iterates $(\bbeta^t, \bnu^t)$ can be produced by the vanilla Condat-V\~u algorithm, the optimized Halpern method~\citep{park2022exact}, or together with the restart technique~\citep{lu2407restarted}.
Moreover, the safe lower bounds share the same convergence rate as the residual error $e^t$.
Thus, if the base algorithm producing $\{\bz^t\}$ is the vanilla Condat-V\~u algorithm, the safe lower bounds converge sublinearly in $O(\frac{1}{\sqrt{t}})$; if the base algorithm is the optimized Halpern method with restart, then the safe lower bounds often converge linearly in a local neighborhood of $\bz^\star$.
\end{remark}

We prove Theorem~\ref{theorem:convergence_rate} by first stating several lemmas.
All proofs of the lemmas can be found in Appendix~\ref{app:methodology-proofs}.

\begin{lemma}[Valid Lower Bounds]
\label{lemma:lower_bounds_are_valid}
Let $\bbeta^\star$ and $(\bzeta^\star,\bu^\star,\bnu^\star)$ be a pair of primal--dual optimal solutions.
Define the conjugate Bregman terms as
\begin{align}
B_f(\bzeta)
&:=f^*(\bzeta)-f^*(\bzeta^\star)
- \ip{\bA\bbeta^\star}{\bzeta-\bzeta^\star},\\
B_g(\bu)
&:=g^*(\bu)-g^*(\bu^\star)
-\ip{\bbeta^\star}{\bu-\bu^\star},\\
B_h(\bnu)
&:=h^*(\bnu)-h^*(\bnu^\star)
-\ip{\bK\bbeta^\star}{\bnu-\bnu^\star},\\
B(\bzeta,\bu,\bnu)
&:=B_f(\bzeta)+B_g(\bu)+B_h(\bnu).
\end{align}
Then, these terms are nonnegative.
For Algorithm~\ref{alg:compute_safe_lower_bound}, the repaired lower bound satisfies
\begin{equation*}
\begin{aligned}
\Delta_{\text{repair}}\left(\bz^t\right) =B(\bzeta^+,\tilde{\bu},\bnu^+)\geq0.
\end{aligned}
\end{equation*}
Under Assumption 2 in Theorem~\ref{theorem:convergence_rate}, the violation-based lower bound satisfies
\begin{equation*}
\begin{aligned}
&\Delta_{\text{violation}}\left(\bz^t\right) \\
=&B(\bzeta^+,\bu^+,\bnu^+) + \sigma_{\calD-\bbeta^\star}\!\left(-V(\bzeta^+,\bu^+,\bnu^+)\right)\geq0,
\end{aligned}
\end{equation*}
where $\calD-\bbeta^\star:=\{\bbeta-\bbeta^\star:\bbeta\in\calD\}$.
\end{lemma}

\begin{lemma}[Surrogate Lower Bounds]
\label{lemma:surrogate_lower_bounds}
Let $\bbeta^\star$ and $B$ be as in Lemma~\ref{lemma:lower_bounds_are_valid}, and let $(\bzeta^+,\bu^+,\bnu^+)$ be the tuple constructed by Algorithm~\ref{alg:compute_safe_lower_bound} at $\bz^t$.
Under Assumption 1 in Theorem~\ref{theorem:convergence_rate},
\begin{equation}
\begin{aligned}
&\Delta_{\text{repair}}\left(\bz^t\right) \\
\leq & B(\bzeta^+, \bu^+,\bnu^+)
+C_{\text{repair}} \norm{V(\bzeta^+,\bu^+,\bnu^+)}_2,
\end{aligned}
\end{equation}
where $C_{\text{repair}} := L_{g^*}+\norm{\bbeta^\star}_2$.

Under Assumption 2 in Theorem~\ref{theorem:convergence_rate},
\begin{equation}
\begin{aligned}
&\Delta_{\text{violation}}\left(\bz^t\right)\\
\leq &B(\bzeta^+,\bu^+,\bnu^+) + R_{\bbeta}\norm{V(\bzeta^+,\bu^+,\bnu^+)}_2,
\end{aligned}
\end{equation}
where $R_{\bbeta}:=\sup_{\bbeta\in\calD}\norm{\bbeta-\bbeta^\star}_2$.
\end{lemma}

\begin{lemma}[One-step Bregman Estimate]
\label{lemma:one_step_Bregman_estimate}
Define
\begin{align}
\bz^{+}:=(\bbeta^+,\bnu^+), \, \br^t:=\bz^t-\bz^{+}, \,\br_{\bbeta}^t:=\bbeta^t-\bbeta^+.
\end{align}
Let $L:=L_f\norm{\bA}_2^2$. Then, the following estimates hold:
\begin{enumerate}[label=(\alph*)]
\item The dual equality violation satisfies
\begin{equation}
\norm{V(\bzeta^+,\bu^+,\bnu^+)}_2
\leq\frac{1}{\sqrt{\sigma}}\norm{\br^t}_M.
\label{eq:one-step-violation-estimate}
\end{equation}
\item The conjugate Bregman terms satisfy
\begin{equation}
\begin{aligned}
&B(\bzeta^+,\bu^+,\bnu^+)\\
\leq & \frac12\Bigl(
\norm{\bz^t-\bz^\star}_M^2-\norm{\bz^{+}-\bz^\star}_M^2 -\norm{\br^t}_M^2\Bigr) \\
&+\frac{L}{2}\norm{\br_{\bbeta}^t}_2^2.
\end{aligned}
\label{eq:one-step-bregman-estimate}
\end{equation}
\item If $\norm{\bz^t-\bz^\star}_M \leq R_{\bz}$, then
\begin{equation}
B(\bzeta^+,\bu^+,\bnu^+)
\leq R_{\bz} e^t+\frac{L_f \norm{A}_2^2}{2s} \left(e^t \right)^2.
\label{eq:one-step-bregman-residual-estimate}
\end{equation}
\end{enumerate}
\end{lemma}

Equipped with these useful lemmas, we are now ready to prove Theorem~\ref{theorem:convergence_rate}.

\begin{proof}[Proof of Theorem~\ref{theorem:convergence_rate}]
The non-negativity of $\Delta_{\text{repair}} \left( \bz^t \right)$ and $\Delta_{\text{violation}} \left( \bz^t \right)$ follows from Lemma~\ref{lemma:lower_bounds_are_valid}.

For the upper bounds, combining Lemma~\ref{lemma:surrogate_lower_bounds} and Lemma~\ref{lemma:one_step_Bregman_estimate} (Parts (a) and (c)) yields
\begin{equation*}
\begin{aligned}
&\Delta_{\text{repair}}(\bz^t)\\
\leq & B(\bzeta^+,\bu^+,\bnu^+)\\
& +\bigl(L_{g^*}+\norm{\bbeta^\star}_2\bigr) \norm{V(\bzeta^+,\bu^+,\bnu^+)}_2\\
\leq & R_{\bz}e^t +\frac{L_f\norm{\bA}_2^2}{2s}(e^t)^2 +\frac{L_{g^*}+\norm{\bbeta^\star}_2}{\sqrt{\sigma}}e^t\\
= & \left( R_{\bz} +\frac{L_{g^*}+\norm{\bbeta^\star}_2}{\sqrt{\sigma}} \right) e^t +\frac{L_f\norm{\bA}_2^2}{2s}(e^t)^2,
\end{aligned}
\end{equation*}
and
\begin{equation*}
\begin{aligned}
&\Delta_{\text{violation}}(\bz^t)\\
\leq & B(\bzeta^+,\bu^+,\bnu^+) + R_{\bbeta}\norm{V(\bzeta^+,\bu^+,\bnu^+)}_2\\
\leq & R_{\bz}e^t
+\frac{L_f\norm{\bA}_2^2}{2s}(e^t)^2
+\frac{R_{\bbeta}}{\sqrt{\sigma}}e^t\\
= &\left(R_{\bz}+\frac{R_{\bbeta}}{\sqrt{\sigma}}\right)e^t +\frac{L_f\norm{\bA}_2^2}{2s}(e^t)^2.
\end{aligned}
\end{equation*}
\end{proof}

\subsection{Matrix-Free Fenchel Conjugate of Quadratic Loss and Beyond}

Although we now have a way to compute a safe lower bound, it is also important to compute it efficiently.
A crucial step is in evaluating $f^*(\bzeta)$, which poses  two potential challenges.
One challenge is that evaluating $f^*(\bzeta)$ could be costly; the other challenge is that it may be difficult to derive a closed-form formula for $f^*$.

Let us take the quadratic loss function as an example.
Suppose $f$ has the following form:
\[
f(\by)=\tfrac12\by^\top\bW\by-\bb^\top\by,
\qquad \bW\in\mathbb R^{n\times n},\quad \bW\succeq0.
\]
Its Fenchel conjugate $f^*$ can be written as follows:
\begingroup
\renewcommand{\arraystretch}{1.2}
\[
f^*(\bzeta)
=
\left\{\begin{array}{@{}l@{\;\;}l@{}}
\displaystyle
\frac12(\bzeta+\bb)^\top\bW^\dagger(\bzeta+\bb),
&
\bzeta+\bb\in\operatorname{Range}(\bW),
\\[6pt]
+\infty,
&
\bzeta+\bb\notin\operatorname{Range}(\bW),
\end{array}\right.
\]
\endgroup
where $\bW^\dagger$ is the pseudoinverse.
Although the closed-form formula exists, finding the pseudoinverse requires doing matrix factorization (or equivalently solving a linear system), which has complexity $O(n^3)$ and is therefore unscalable.

However, we can actually compute the Fenchel conjugate without doing matrix factorization.
The insight lies in exploiting Line~\ref{alg_line:dual_variable_zeta_update} in Algorithm~\ref{alg:compute_safe_lower_bound} and the Fenchel--Young inequality~\citep[Theorem~23.5]{rockafellar1970convex}, which tells us that
\begin{equation}
\label{eq:compute_f*_using_f}
    f^*(\bzeta) = \ip{\bzeta}{\by} - f(\by) \quad \text{if} \quad \bzeta = \nabla f(\by).
\end{equation}
Thus, we do not need to do any matrix factorization, nor we do even need to know or derive the closed-form formula for $f^*$ when computing safe lower bounds.
In the special case of the quadratic loss function mentioned at the beginning, the Fenchel conjugate is simply $f^*(\bzeta) = \frac{1}{2} \by^\top \bW \by$.
This is an important computational trick to keep in mind as we demonstrate its scalability advantage in the experimental section.

\subsection{Early Termination of First-order Method}
\label{sec:Early_Termination_of_First_order_Method}

\begin{algorithm}[tbp]
\caption{Compute Safe Upper bound}
\label{alg:compute_safe_upper_bound}
\LinesNumbered
\SetCommentSty{AlgorithmCommentStyle}
\DontPrintSemicolon
\KwIn{Primal iterate $\bbeta^t\in\dom(f\circ\bA)\cap\dom g$ and a node-feasible anchor $\ba_{\calN}\in\dom p$. }
\KwOut{Safe upper bound $UB^t$ for Problem~\eqref{eq:node-relaxation}.}
\For{$l=0,\ldots,m$}{
    \tcp{Find the minimum convex combination parameter so that $l$-th constraint is obeyed; see Appendix~\ref{app:parallel-repair} for closed-form solutions for various standard constraints}
    $\alpha_l \gets \min \{\alpha \in [0, 1]:  \left(1-\alpha \right) \bbeta^t + \alpha \ba_{\calN}  \in \calC_l \}$\;
}
\tcp{Find the minimum convex combination parameter so that all constraints are obeyed; }
$\alpha \gets \max \left( \{ \alpha_l \}_{l=0}^m \right)$\;
$\bbeta^{\alpha} \gets \left(1-\alpha \right) \bbeta^t + \alpha \ba_{\calN}$\;
$UB^t \gets f(\bA \bbeta^{\alpha}) + g(\bbeta^{\alpha})$\;
\KwRet{$UB^t$}\;
\end{algorithm}

As mentioned in Section~\ref{subsec:convex-relaxation}, we can terminate the FOM early whenever $LB^t \geq U^{\text{incum}}$, where $U^{\text{incum}}$ is the value of the incumbent~\citep{liang2020early}.
This is because $p_{\calN}^\star \geq LB^t \geq U^{\text{incum}}$, so no solution in that partitioned region can improve the incumbent.
Thus, we should stop solving the relaxation and do pruning.

However, when $U^{\text{incum}} > LB^t$, the safe lower bound offers little help in accelerating the overall computation.
Instead, what becomes useful is to obtain a~\emph{safe upper bound}, $UB^t$.
Whenever $U^{\text{incum}} > UB^t $, we can also terminate the FOM early.
This is because $UB^t \geq p_{\calN}^\star$ implies $U^{\text{incum}} > p_{\calN}^\star$, so the relaxation is not strong enough to prune the node.
Therefore, we could stop solving this relaxation and branch, potentially delaying an incumbent improvement until later in the tree.
The next section reports the resulting savings in FOM iterations.

Algorithm~\ref{alg:compute_safe_upper_bound} shows the details of how to obtain safe upper bounds in presence of infeasible primal iterates.
A repaired point updates the mixed-integer incumbent only if it also satisfies integrality; otherwise, its objective value bounds only the continuous node relaxation from above.
Under additional assumptions of Appendix~\ref{app:primal-convergence},
we can show that $\bbeta^t\to\bbeta^\star$ implies $UB^t\to p_{\calN}^\star$.
Together, $LB^t$ and $UB^t$ provide a two-sided early termination rule.

\section{Experiments}
\label{sec:experiment}

In this experimental section, we address four questions:
a) what advantages do safe lower and upper bounds bring to running a first-order method when solving node relaxations?
b) how much computational saving can the matrix-free conjugate evaluation provide?
c) how fast does our customized FOM-based BnB solver compare with commercial MIP solvers?
d) how do primal and dual feasibility repairs affect branch and bound running time?

In Appendix~\ref{app:additional-experiments}, we also provide additional results on node-relaxation benchmarks, BnB experiments on synthetic datasets, ablation study of incorporating safe lower and upper bounds, and comparison of $LB_{\text{repair}}$ and $LB_{\text{violation}}$ on synthetic node relaxations.

\begin{figure}[t]
    \centering
\includegraphics[width=0.60\linewidth]{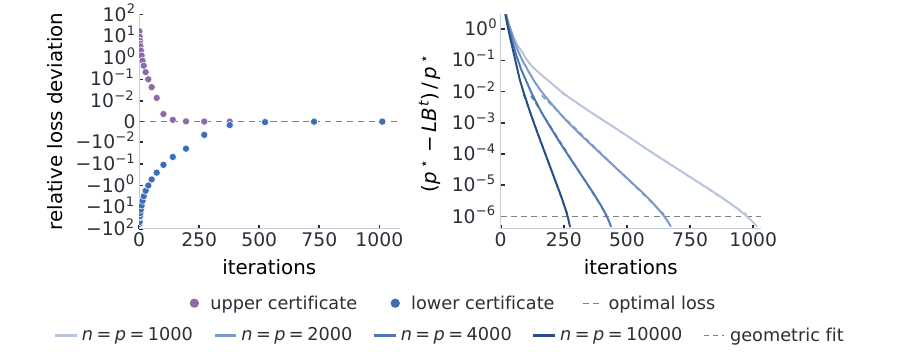}
    \caption{
        Left: sampled safe lower and upper bounds for synthetic sparse linear regression with $n=p=1000$, shown as relative gaps from the optimal loss, which is reset to $0$;
        Right: convergence of dual gaps across various problem sizes.
    }
    \label{fig:relaxation-feasibility}
\end{figure}

\begin{figure*}[!b]
    \centering
    \includegraphics[width=0.95\linewidth,trim=0 10pt 0 4pt,clip]{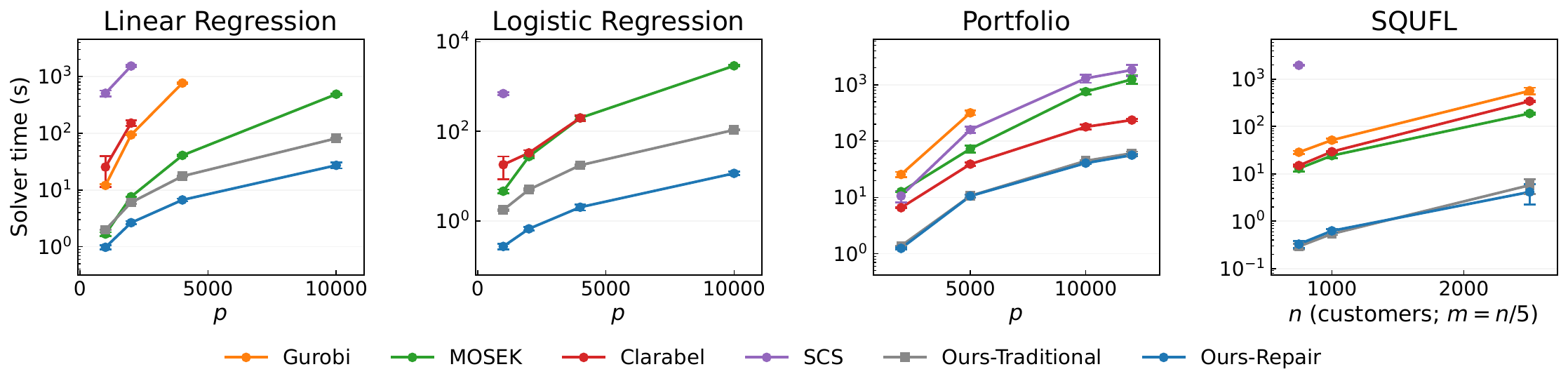}
    \setlength{\abovecaptionskip}{2pt}
    \caption{
        Time (on CPU) to solve a single node relaxation on synthetic instances as a function of problem size.
        Ours$+$Repair can help terminate the algorithm early based on primal-dual gap (left two plots) and does not add much computational overhead (right two plots).
    }
    \label{fig:noderelaxations}

\end{figure*}

\paragraph{Experimental Setup.}
Across synthetic and real-world data, we study (i) sparse linear and logistic regression with at most $k$ nonzeros, box constraints, and affine constraints $\bm C\beta\leq b_{\mathrm{aff}}$~\cite{liu2026gpu}; (ii) separable quadratic uncapacitated facility location (SQUFL), where demand is assigned to open facilities with quadratic costs~\cite{gunluk2007minlp}; and (iii) ridge-regularized sparse mean-variance portfolios with at most $k$ nonzeros and budget constraints, plus a return constraint for market-data instances~\cite{bertsimas2022scalable}. Appendix G.2 describes the data, generation, preprocessing, and hardware.

\subsection{Advantages of Safe Lower and Upper Bounds for Solving Node Relaxations}

As mentioned in the Methodology section, with primal and dual repairs, we can extract safe lower and upper bounds from the node relaxation problem at any iteration.
This is shown in the left plot of Figure~\ref{fig:relaxation-feasibility}, which solves the relaxation of a synthetic sparse linear regression problem with $n=p=1000$.
In the right plot of Figure~\ref{fig:relaxation-feasibility}, we show the empirical convergence of the safe lower bounds across various problem sizes.
The empirical convergence rate is consistent with our Theorem~\ref{theorem:convergence_rate} and the theoretical convergence rate of the residual error $e^t$, which converges sub-linearly globally and often converges linearly locally~\citep{liang2018local}.

For some large-scale problems, safe lower and upper bounds can help us terminate the FOM early as it provides a new stopping criterion -- primal-dual gap.

We next assess scalability by measuring the time required to solve increasingly large node relaxations across all four problem classes. As shown in Figure~\ref{fig:noderelaxations}, the traditional Condat--V\~u is already orders-of-magnitude faster than off-the-shelf first-order and second-order solvers because it avoids linear-system solves as a subroutine.
However, we can run even faster.
For traditional Condat--V\~u method, it only terminates once the residual error is below the tolerance level;
by using the primal-dual gap as an additional stopping criterion, we obtain further run up to one order of magnitude faster on the regression instances.
For problems where this additional stopping criterion is not helpful, the two plots on the right show that the extra costs of computing safe lower and upper bounds are negligible and do not add much overhead.

\begin{table}[!h]
\centering
\setlength{\belowcaptionskip}{6pt}
\caption{CPU time (ms) for 64 conjugate evaluations. Mean $\pm$ sample SD across three seed medians, each from five repetitions. Explicit-inverse setup time is excluded; timed batches include the required products.}
\label{tab:portfolio-dual-scaling}
\normalsize
\setlength{\tabcolsep}{6pt}
\renewcommand{\arraystretch}{1.25}
\begin{tabular*}{0.90\linewidth}{@{\extracolsep{\fill}}lrrrr@{}}
\toprule
$p$ & 1k & 2k & 4k & 10k \\
\midrule
matrix-free & $\mathbf{1.100}\pm0.014$ & $\mathbf{1.469}\pm0.017$ & $\mathbf{2.371}\pm0.019$ & $\mathbf{5.543}\pm0.085$ \\
Explicit $\bW^{-1}$ & $32.932\pm0.468$ & $169.748\pm0.180$ & $678.865\pm25.398$ & $3927.710\pm6.723$ \\
\bottomrule
\end{tabular*}

\end{table}

\subsection{Benefit of Matrix-free Fenchel Conjugate}

We next validate the benefit of our matrix-free method when computing the Fenchel conjugate of the loss function, $f^*$.

\begin{table*}[!t]
\centering
\setlength{\belowcaptionskip}{4pt}
\caption{
    BnB results on solving problems with discrete variables, with a relative gap tolerance of $10^{-4}$.
    TL denotes the 3600\,s time limit; -- denotes an unavailable or unverified value. A gap of $0.00$ denotes attainment of the prescribed tolerance.
}
\label{tab:okglm-bnb-wide}
\fontsize{8.5}{10}\selectfont
\setlength{\tabcolsep}{1.0pt}
\renewcommand{\arraystretch}{1.0}%
\begin{tabular*}{\linewidth}{@{\extracolsep{\fill}}l*{15}{r}@{}}
\toprule
& \multicolumn{3}{c}{Gurobi}
& \multicolumn{3}{c}{MOSEK}
& \multicolumn{3}{c}{ours CPU}
& \multicolumn{3}{c}{ours GPU}
& \multicolumn{3}{c}{ours GPU (Batched)} \\
\cmidrule(lr){2-4} \cmidrule(lr){5-7}
\cmidrule(lr){8-10} \cmidrule(lr){11-13} \cmidrule(lr){14-16}
& $t$ (s) & gap (\%) & nodes
& $t$ (s) & gap (\%) & nodes
& $t$ (s) & gap (\%) & nodes
& $t$ (s) & gap (\%) & nodes
& $t$ (s) & gap (\%) & nodes \\
\midrule
SANTANDER
& TL & -- & --
& TL & 1.69 & 3.0
& TL & 0.063 & 3187.0
& 940.8 & 0.00 & 4653.0
& \textbf{74.2} & 0.00 & 6513.0 \\
DOROTHEA
& TL & 45.58 & 1.0
& TL & -- & --
& 169.2 & 0.00 & 23.0
& 6.9 & 0.00 & 23.0
& \textbf{5.8} & 0.00 & 31.0 \\
YF-27907
& TL & -- & --
& TL & -- & --
& TL & 13.13 & 88.0
& 1830.0 & 0.00 & 1627.0
& \textbf{175.6} & 0.00 & 1633.0 \\
YF-35629
& TL & -- & --
& TL & 13.21 & 17.0
& TL & 9.30 & 266.0
& 2597.4 & 0.00 & 2447.0
& \textbf{248.3} & 0.00 & 2509.0 \\
squfl200-1000
& 1706.0 & 0.00 & 70.0
& 401.6 & 0.00 & 8.0
& 20.4 & 0.00 & 5.7
& 10.2 & 0.00 & 5.7
& \textbf{1.3} & 0.00 & 6.3 \\
squfl500-2500
& TL & 0.32 & 3.3
& 2729.7 & 0.00 & 3.0
& 166.4 & 0.00 & 3.7
& 13.1 & 0.00 & 3.7
& \textbf{1.5} & 0.00 & 4.3 \\
\bottomrule
\end{tabular*}
\end{table*}

We measure this computational saving on portfolio instances by comparing our formula with an explicit-inverse implementation, with a total of 64 conjugate evaluations per instance.
As shown in Table~\ref{tab:portfolio-dual-scaling}, the advantage of our method grows with the problem size: at $p=10{,}000$, our approach took only $5.54$~ms whereas the explicit-inverse implementation took $3.93 \times 10^3$~ms. Thus, once the gradient argument is available, the conjugate can be evaluated cheaply without materializing, factorizing, or inverting the dense quadratic matrix.

\subsection{Comparison with BnB Solvers}
Faster node relaxations do not necessarily imply faster branch-and-bound solves, since total runtime also depends on the search tree.
We next develop a customized BnB solver based on our node relaxation techniques and compare it with state-of-the-art MIP solvers on four real-world and two large synthetic SQUFL instances. The results are shown in Table~\ref{tab:okglm-bnb-wide}.

When all methods are run on CPUs, our method is already 1--2 orders of magnitude faster than Gurobi and MOSEK. For challenging instances where none of the methods can close the optimality gap, our method produces the smallest gap. Furthermore, unlike Gurobi and MOSEK, our method is GPU-friendly as it does not require adding cutting planes nor solving linear systems. The most computationally intensive part is matrix-vector multiplication, which can be easily accelerated on GPUs. When we run our method on GPUs, we can close the optimality gap on all instances within the specified time limit (1h). The running time is one order of magnitude (up to 24.5$\times$) faster than running on CPUs when processing nodes sequentially and achieves another order of magnitude (up to 12.7$\times$) when processing multiple nodes simultaneously in the batched mode.

\subsection{Ablation Study of Feasibility Repairs}

In order to study the individual effects of primal and dual repairs for BnB, we conduct an ablation study by comparing how many iterations it takes to solve node relaxations.
We fix a pre-certified incumbent objective value and evaluate all stopping rules on the same node trajectories.
Each node is solved to convergence, while we record at which iteration either the primal or the dual repair would have triggered the termination of the relaxation, enabling us to isolate the contribution of our two restoring routines.

\begin{minipage}{\linewidth}
The results are shown in Table~\ref{table:repair_ablation}.
Both repairing methods are effective:
dual repair enables pruning nodes early, while primal repair helps to avoid unnecessary iterations of solving node relaxations to high precision.
Together, they provide the strongest performance boost across the test instances.
\end{minipage}

\begin{table}[!htb]
\centering
\setlength{\belowcaptionskip}{6pt}
\caption{
    Average number of iterations per node.
    Parentheses report the iteration count relative to no repair.
}
\normalsize
\setlength{\tabcolsep}{6pt}
\renewcommand{\arraystretch}{1.25}%
\begin{tabular*}{0.90\linewidth}{@{\extracolsep{\fill}}lrrrr@{}}
\toprule
Instance & no repair & primal only & dual only & both \\
\midrule
SANTANDER      & 262.8  & 162.0\,(62\%)  & 117.9\,(45\%)  & \textbf{17.1}\,(7\%) \\
DOROTHEA       & 101.5  & 86.5\,(85\%)   & 21.3\,(21\%)   & \textbf{6.3}\,(6\%) \\
YF-27907       & 5801.6 & 583.3\,(10\%)  & 5355.7\,(92\%) & \textbf{137.5}\,(2\%) \\
YF-35629       & 9424.0 & 1003.4\,(11\%) & 8569.2\,(91\%) & \textbf{148.6}\,(2\%) \\
squfl200-1000  & 64.1   & 50.3\,(78\%)   & 39.8\,(62\%)   & \textbf{26.0}\,(41\%) \\
squfl500-2500  & 86.0   & 64.4\,(75\%)   & 51.0\,(59\%)   & \textbf{29.3}\,(34\%) \\
\bottomrule
\end{tabular*}

\label{table:repair_ablation}
\end{table}

\section{Conclusion}
\vspace{2mm}

We have introduced the first primal-dual repair framework that allows first-order methods to solve node relaxations within branch-and-bound.
By repairing infeasible primal and dual iterates, our method provides safe lower and upper bounds at every iteration, enabling node pruning, incumbent updates, and early termination without requiring the wait of convergence when solving node relaxations.
A matrix-free evaluation of Fenchel conjugates further reduces the computational cost for quadratic and other structured losses.
On synthetic and real-world benchmarks, our method is up to two orders of magnitude faster than commercial MIP solvers, with an additional one to two orders of magnitude speedup on GPUs.
This opens doors to developing new generic MIP solvers based on first-order methods and modern parallel computing hardware.

\setlength{\parskip}{3pt}
\section*{Acknowledgments}

We thank the Department of Computer Science and Engineering at the University of Bologna for providing GPU computing resources, and Prof.~Maurizio Gabbrielli for granting access to the cluster.
We also thank Daniel~Bienstock for his helpful advice on portfolio optimization.
This work used the Delta system at the National Center for Supercomputing Applications through allocation CIS250029 from the Advanced Cyberinfrastructure Coordination Ecosystem: Services \& Support (ACCESS) program, which is supported by National Science Foundation grants \#2138259, \#2138286, \#2138307, \#2137603, and \#2138296.
The authors recognize the Penn State Institute for Computational and Data Sciences (ICDS)  (RRID:SCR\_025154) for providing access to computational research infrastructure (RRID:SCR\_026424).

\clearpage
\bibliographystyle{plainnat}
\ifdefined\RootBuild
  \bibliography{manuscript/references}
\else
  \bibliography{references}
\fi

\clearpage
\thispagestyle{empty}
\onecolumn

\input{content/07_Appendix}

\label{app:instances}

\end{document}